\documentclass[11pt]{amsart}
\usepackage{color}
\usepackage{mathrsfs}
\usepackage{hyperref}
\usepackage{amsmath, amsthm, amssymb}
\usepackage{geometry}                
\usepackage[english]{babel}
\usepackage[utf8]{inputenc}
\usepackage[parfill]{parskip}    
\usepackage{graphicx}
\usepackage{amssymb}
\usepackage{epstopdf}
\newtheorem{teo}{Theorem}[section]
\newtheorem{lema}[teo]{Lemma}

\newtheorem{cor}[teo]{Corollary}

\newtheorem{defi}[teo]{Definition}
\newtheorem{remark}[teo]{Remark}

\newcommand{\reais}[1]{\mathbb R^{#1}}

\newcommand{\inte}[1]{\mathbb Z^{#1}}
\newcommand{\toro}[1]{\mathbb T^{#1}}

\newcommand{\vertiii}[1]{{\left\vert\kern-0.25ex\left\vert\kern-0.25ex\left\vert #1 
    \right\vert\kern-0.25ex\right\vert\kern-0.25ex\right\vert}}
\newcommand{\ar}{\rightarrow}
\newcommand{\twopartdef}[4]
{
	\left\{
		\begin{array}{ll}
			#1 & \mbox{,  if } #2 \\
			#3 & \mbox{, if } #4
		\end{array}
	\right.
}

\newcommand{\cone}[4]
{
\mathcal C^{#1}_{#2}(#3,#4)
}

\title{Central Lyapunov exponent of partially hyperbolic diffeomorphisms of $\toro{3}$}
\author{G. Ponce}
\address{Departamento de Matem\'atica,
  ICMC-USP S\~{a}o Carlos-SP, Brazil.}
  \email{gaponce@icmc.usp.br}
\author{A. Tahzibi} 
\address{Departamento de Matem\'atica,
  ICMC-USP S\~{a}o Carlos-SP, Brazil.}
\email{tahzibi@icmc.usp.br}
\date{}                                          
\thanks{G. Ponce is enjoying a Doctoral scholarship of FAPESP. A. Tahzibi had the support of CNPq and FAPESP}

\begin{document}
\maketitle

\begin{abstract}
In this paper we construct some ``pathological" volume preserving partially hyperbolic diffeomorphisms on $\toro{3}$ such that their behaviour in small scales in the central direction (Lyapunov exponent) is opposite to the behavior of their linearization. These examples are isotopic to Anosov. We also get partially hyperbolic diffeomorphisms isotopic to Anosov (consequently with non-compact central leaves) with zero central Lyapunov exponent at almost every point.
\end{abstract}


\section{Introduction and statement of the result}
The ergodic theory of ``beyond uniformly hyperbolic dynamics" is an extensive research area and has connection to many other topics.  Partial hyperbolicity is a form of relaxing the uniform hyperbolicity condition with natural interesting examples (see  \cite{BPSW},  \cite{HHUsurvey}).  One of the amazing issues raising in the study of ergodic properties of partially hyperbolic dynamics is the existence of invariant foliations and their topological and metric properties. A complete comprehension of invariant foliations by partially hyperbolic dynamics is also an important tool for the classification 
of these dynamics and the manifolds which support them.

In this paper we introduce some new ``pathological" examples of partially hyperbolic diffeomorphisms. We study the relationship between central Lyapunov exponents and topology of leaves of central foliation of a partially hyperbolic diffeomorphism and its linearization. More precisely, we find an open set of partially hyperbolic diffeomorphisms $f: \mathbb{T}^3 \rightarrow \mathbb{T}^3$ isotopic to linear Anosov diffeomorphisms $A$ such that the central Lyapunov exponent of $f$ is positive almost everywhere while the central bundle of $A$ is contracting. This opposite behavior in the asymptotic growth (manifested by the sign of Lyapunov exponent) which is a local issue contrasts with the compatible behavior in the large scale between $f$ and $A.$ (See Preliminaries Section.)

We also obtain examples of partially hyperbolic diffeomorphisms isotopic to Anosov with non compact central leaves and zero central Lyapunov exponent almost everywhere. This is also a contrast between  the global topology of central leaves of non linear and linearization of a partially hyperbolic diffeomorphism. . 

\begin{teo} \label{teo1}

There exists an open set of volume preserving partially hyperbolic diffeomorphisms $U$ such that for any $f \in U$ almost every point has positive Lyapunov exponent and the linearization of $f$ is an Anosov diffeomorphism with splitting $E^u \oplus E^s \oplus E^{ss}.$
\end{teo}

\begin{teo} \label{teo2} There exist volume preserving partially hyperbolic diffeomorphism $f:\toro{3} \ar \toro{3}$ (isotopic to linear Anosov automorphism) with zero central Lyapunov exponent for Lebesgue almost every point of $\toro{3}$ and non-compact central leaves.
\end{teo}

There are important questions on our pathological examples which are not known by us for the moment. 
We mention that by a recent result of A. Hammerlindl and R. Ures \cite{HU}, a non-ergodic volume preserving isotopic to Anosov diffeomorphism on $\mathbb{T}^3$, if it exists, should have zero central Lyapunov exponent almost everywhere. 
The ergodicity of the diffeomorphisms with zero central exponent in Theorem \ref{teo2} is an open problem.

Another interesting issue is related to absolute continuity of central foliation. We do not know whether the central foliation of diffeomorphisms with zero central exponent in Theorem \ref{teo2} is absolutely continuous (see \cite{PVW} for a survey about absolute continuity) or not. We believe that it is not the case and formulate the following question.

{\bf Question 1: } Let $f:M \ar M$ be a partially hyperbolic diffeomorphism of $M=\toro{3}$ with absolutely continuous central foliation $\mathcal F^c$ and central Lyapunov exponent equal to zero at Lebesgue almost every point ($\lambda^c = 0$ Lebesgue-a.e). Is it true that the leaves of $\mathcal F^c$ are compact? 

We remark that A. Tahzibi and F. Micena \cite{MT} recently gave an affirmative answer to this question assuming  $\mathcal F^c$ satisfies a uniformly bounded density condition which is a regularity condition stronger than leafwise absolute continuity.


The idea of the proof of Theorems \ref{teo1} and \ref{teo2} is to take a family of hyperbolic linear automorphisms $f_k: \toro{3} \ar \toro{3}$ with eigenvalues $\lambda^s_k < \lambda^c_k < 1 < \lambda^u_k$ in such a way that $$\lambda^s_k \ar 0, \lambda^c_k \ar 1, \lambda^u_k \ar \infty$$
as $k \ar \infty $ and moreover the corresponding unitary eigenvectors converge to a fixed orthonormal basis. Then we apply Baraviera-Bonatti \cite{BB} method of local perturbations and by the choice of the Anosov automorphisms, we are able to show that this local perturbation yields a new partially hyperbolic diffeomorphism with positive central Lyapunov exponent in average. By continuity argument we find some isotopic to Anosov and partially hyperbolic diffeomorphism with vanishing integral of central Lyapunov exponent. 

Finally, by  \cite{HU}  the diffeomorphisms obtained in Theorem \ref{teo1} are ergodic and so almost every point have the same Lyapunov exponent. For diffeomorphisms obtained in Theorem \ref{teo2}, with vanishing average of central exponent, without proving the ergodicity we obtain vanishing Lyapunov exponent almost everywhere.

\section{Preliminaries}

Let $M$ be a compact smooth manifold. A diffeomorphism $f: M \ar M$  is partially hyperbolic if there exists a $Df$-invariant splitting of the tangent bundle
$$TM=E^s \oplus E^c \oplus E^u$$ such that $Df$ uniformly expands all vectors in $E^u$ and uniformly contracts all vectors in $E^s.$ While vectors in $E^c$ are neither contracted as strongly as any nonzero vector in $E^s$ nor expanded as strongly as any nonzero vector in $E^u.$ $f$ is called {\bf absolutely partially hyperbolic} if the domination property between the three mentioned sub bundles is uniform on the whole manifold, i.e. there are constants $a, b > 0$ such that for all $x \in M$ and any unit vectors $v^*, * \in \{s, c, u\}$ in $T_x M$
$$
 \|D_xf(v^s)\| < a < \|D_xf (v^c)\| < b < \|D_xf(v^u)\|.
$$
Absolutely partial hyperbolicity conditions can be expressed equivalently in terms of invariant cone families. In the appendix we make some precise statements. 
In this paper we always deal with absolutely partially hyperbolic diffeomorphisms. 

For all partially hyperbolic diffeomorphisms, there are foliations $\mathcal F^{\tau}, \tau=s,u$ tangent to the sub-bundles $E^{\tau}, \tau=s,u$ called stable and unstable foliation respectively. On the other hand, the integrability of the central sub-bundle $E^c$  is a subtle issue and is not the case in general partially hyperbolic setting (see \cite{HHU}). By M. Brin, D. Burago, S. Ivanov \cite{BBI} all absolutely partially hyperbolic diffeomorphisms on $\toro{3}$ admit a central foliation tangent to $E^c.$ 

Let $f: \toro{3} \rightarrow \toro{3}$ be a partially hyperbolic diffeomorphism. Consider $f_* : \mathbb{Z}^3 \rightarrow \mathbb{Z}^3$ the action of $f$ on the fundamental group of $\toro{3}.$ $f_*$ can be extended to $\mathbb{R}^3$ and the extension is the lift of  a unique linear automorphism $A : \toro{3} \rightarrow \toro{3}$ which is called the linearization of $f.$ It can be proved that $A$ is a partially hyperbolic automorphism of torus (\cite{BBI}). It is not difficult to see that in large scale $f$ and $A$ behaves similarly (see \cite{Htese}, corollary $2.2$). More precisely, for each $k \in \mathbb{Z}$ and $C > 1$ there is an $M > 0 $ such that for all $x, y \in \mathbb{R}^3$,
$$
 \|x - y\| > M \Rightarrow \frac{1}{C} < \frac{\|\tilde{f}^k(x) - \tilde{f}^k(y)\|}{\|A^k(x) - A^k(y)\|} < C.
$$ where $\tilde{f}: \mathbb{R}^3 \ar \mathbb{R}^3$ is the lift of $f$ to $\mathbb{R}^3.$
The examples in the open set $U$ of Theorem \ref{teo1} shows that in the infinitesimal scale opposite behaviors can occur.

A. Hammerlindl proves that any absolutely partially hyperbolic diffeomorphism $f$ on $\toro{3}$ is leaf conjugated to its linearization (for higher dimensions see \cite{H}). In particular the central leaves of $f$ are all homeomorphic.

It is easy to see that a linear partially hyperbolic diffeomorphism of $\mathbb{T}^3$ is either Anosov or all of the leaves of $\mathcal F^c$ are compact, i.e, homeomorphic to $\mathbb{S}^1.$ In the latter case the central eigenvalue is equal to one.   In  theorem \ref{teo2} we give example of partially hyperbolic diffeomorphisms with zero central Lyapunov exponent almost everywhere and non compact leaves.


\noindent{\bf Acknowledgement.}  
We would like to thank C. Bonatti for very helpful suggestions (during the conference Beyond Uniform Hyperbolicity 2011- Marseille) for the conjecture of the second named author's talk in that conference. We are gratefull to A. Hammerlindl for elegant mathematical suggestions for the presentation of our results.  We also would like to thank R. Varão for carefully reading the manuscript of this article. 

\section{Local Perturbation} \label{LP}

In this section we describe briefly a local perturbation process introduced by A. Baraviera, C. Bonatti \cite{BB}. For any partially hyperbolic diffeomorphism $f_0$ they construct a $C^1$-arc of diffeomorphisms $\{f_r\}$ for which the integral of central Lyapunov exponent of $f_r$  is strictly bigger than integral of the central exponent of $f_0$. 

In \cite{BB},  this perturbation procedure is made in a general case. Here we will use  the perturbation argument just for the linear case.

 Let  $f:\toro{3} \ar \toro{3}$   be a volume preserving, linear partially hyperbolic diffeomorphism. Denote by $\lambda^s < \lambda^c < \lambda^u$ the eigenvalues of $f$ and its unitary eigenvectors by $e_s,e_c,e_u$ respectively. Thus,  the directions of $e_s,e_c,e_u$ are the directions of the subbundles $E^s,E^c,E^u$. Let $p$ be a non-fixed point. We take a $C^1-$local coordinate system  on a neighborhood $V$ centered at $p$ such that $\{e_s, e_c, e_u\}$ are directed by $\frac{\partial}{\partial x},  \frac{\partial}{\partial y}, \frac{\partial}{\partial z}.$ Moreover, the expression of volume form on $\toro{3}$ coincides with the Lebesgue measure on $\mathbb{R}^3.$
 
Let $B_1(0)$ be the unit ball of $\mathbb{R}^3.$ Given any ball $B_r(p)$ inside $V$ we denote by $\varphi_{r} : B_{r}(p) \ar B_1(0)$ the diffeomorphism which in local coordinates is a homothety of ratio $\frac{1}{r}.$ More precisely, if $\pi: V \rightarrow \mathbb{R}^3$ is the mentioned coordinate system $\varphi_r(x):=  \frac{\pi(x)}{r}.$

 Let $h:B_1(0) \ar B_1(0)$,  $h \ne Id$, a volume preserving diffeomorphism which preserves the $x$- direction, and equal to the identity on a neighborhood of the boundary of $B_1(0)$.\\
We define the diffeomorphism $h_{r}: \toro{3} \ar \toro{3}$ by:

\begin{equation} \label{normalize} h_{r}(w) = \twopartdef {w}{w \notin  B_{r}(p);}{ \varphi_{r} ^{-1} \circ h \circ \varphi_{r} (w) }{ w \in   B_{r}(p).  }   
\end{equation}

Finally, we define the arc of diffeomorphisms $\{f_r\}_{r\in [0,1]}$ by:
\begin{equation} \label{normalize2} f_{r} := f \circ h_{r}.
\end{equation}

Also, we take $h$ to satisfy
$$||h-Id||_{C^1} < 1.$$

 Since $h$ preserves the direction of $e_s$ we can write
\begin{equation} \label{cu-preserve}
Dh(p)e_u = h^u(p)e_u + h^c(p)e_c.
\end{equation}

\begin{lema}\cite{BB}
Let $h$ be as above, then
\begin{displaymath} 
I(h) := \int_{B_1(0)} \log h^u(p) dm(p) <0.
\end{displaymath}
\end{lema}
Consider $n_{r}$ the least positive integer such that
$$f^{n_{r}}(B_r) \cap B_r \ne \emptyset.$$
Denote by $\lambda^u_{r}(p)$ the unstable Lyapunov exponent of $f_r$ at $p$ and define:
$$\sigma^u_{f_{r}} = \int \log J^u_{f_{r}}(p) dm(p), \hspace{.3cm} \sigma^c_{f_{r}} = \int  \log J^c_{f_{r}}(p) dm(p)$$

where $J^{\tau}_{f_r}( p)$ denotes the Jacobian of $f_r$ on $E^{\tau}_{f_r}( p )$, that is, the modulus of the determinant of the restriction of $Df_r (p )$ to $E^{\tau}_{f_r}( p )$, $\tau=s, c,u$.

\begin{lema} \cite{BB}
Let $\sigma^u_{f_{r}}$ and $\sigma^c_{f_{r}}$ be as above, then
$$\log \lambda^u - \sigma^u_{f_r} \geq vol(B_r)(-I(h) - C \alpha^{n_r})$$
where $\alpha = \lambda^c/\lambda^u$, and 
$C = \max_{x\in B_r} \frac{h^c_r}{h^u_r} \cdot \max_{x \in B_r}{|| \text{Proj} _{u}(e_c)||}$
with $Proj_{u}(e_c)$ denoting the projection of $e_c$ over $E^u$ parallel to the new center bundle.

\end{lema}

Observe that by (\ref{cu-preserve}) the perturbation $h_r$ preserves the center unstable bundle and the center-unstable jacobian of $h_r$ is equal to one. So the integral of the logarithm of the jacobian of $f_r$ in the center-unstable direction is the same of the one for $f$ (see \cite{BB}, pg.1664). Thus, we have the following corollary.

\begin{cor} \label{corocentral}
With the previous notations, the difference between $\sigma^c_{f_r}$ and $\log(\lambda^c)$ is bounded from below as follows:
$$\sigma^c_{f_{r}} - \log \lambda^c = \log \lambda^u - \sigma^u_{f_{r}}  \geq vol (B_{r}(p)) \cdot \left(-I(h) - C \cdot \left(\frac{\lambda^c}{\lambda^u}\right)^{n_r}\right).$$
\end{cor}



\section{Family of Anosov Linear Automorphisms}


In order to realize a perturbation that changes the sign of the central Lyapunov exponent, it is reasonable to take a diffeomorphim with central exponent close to $0$ and big unstable exponent  (so that we can borrow some hyperbolicity from the unstable direction). 
For each $k \in \inte{}$ define the linear automorphism $f_k: \toro{3} \ar \toro{3}$ induced by the integer matrix: 
$$A_k=\left( \begin{array}{ccc}
0 & 0 & 1 \\
0 & 1 & -1 \\
-1 & -1 & k \end{array} \right). $$

The characteristic polynomial of $A_k$ is
$$p_{k}(x)= x^3 -(k+1)x^2 + kx -1.$$

\begin{lema}
For all $k \geq 5$, $A_k$ has real eigenvalues $0<\lambda^s_k < \lambda^c_k < 1<\lambda^u_k$ and
$$\lambda^s_k \ar 0, \lambda^c_k \ar 1, \lambda^u_k \ar \infty$$
as $k \ar \infty $.
\end{lema}
\begin{proof}
First of all note that :
\begin{itemize}
\item $p_k(1/2) = \frac{k}{4}-\frac{9}{8} > 0$, $\forall k \geq 5$;
\item $p_k(1) = p_k(k) = -1$, $\forall k$;
\item $p_k(k+1) = k(k+1)-1 \geq 1$, $\forall k \geq 1.$
\end{itemize}

So, for all $k \geq 5$, $p_k$ has a root $\lambda^u_k \in (k,k+1)$ and a root $\lambda^c_k \in (1/2,1)$. Denoting by $\lambda^s_k$ the other root we have: 
$$0<\lambda^s_k = \frac{1}{\lambda^c_k \cdot \lambda^u_k} < \lambda^c_k < 1< k <\lambda^u_k.$$

Now, given any $0<\varepsilon<1$ we have
$$p_k(1-\varepsilon) = k(1-\varepsilon)\varepsilon -\varepsilon(1-\varepsilon)^2 -1$$

which is trivially positive for large values of $k$. 
That is, $\lambda^c_k \ar 1$ when $k \ar \infty$. Also, since $k <\lambda^u_k$ and $\lambda^s_k \cdot \lambda^c_k \cdot \lambda^u_k =1$ we conclude that
$$\lambda^c_k \ar 1, \lambda^u_k \ar \infty, \text{ and } \lambda^s_k \ar 0$$
as $k \ar \infty.$
\end{proof}

Next, we evaluate the stable, central and unstable directions of $f_k$.
$$
A_k \cdot \left(\begin{array}{c}
a \\
b \\
c \end{array}\right) =\left(\begin{array}{c}
\lambda a \\
\lambda b \\
\lambda c \end{array}\right)  \Rightarrow 
\left(\begin{array}{c}
a \\
b \\
c \end{array}\right) = \left(\begin{array}{c}
a \\
\frac{\lambda a}{1-\lambda} \\
\lambda a \end{array}\right). $$
So the directions are $v^{\tau}_k:=(1, \lambda_k^{\tau}/(1-\lambda_k^{\tau}), \lambda_k^{\tau})$ where $\tau=s,c,u$. Let ${\displaystyle e^k_{\tau}:=\frac{v_k^{\tau}}{||v_k^{\tau}||}}$, $\tau=s,c,u$. Then we have:
$$e^k_s \ar \left(\begin{array}{c}
1 \\
0 \\
0 \end{array}\right),
e^k_c \ar \left(\begin{array}{c}
0 \\
1 \\
0 \end{array}\right),
e^k_u \ar \left(\begin{array}{c}
0 \\
0 \\
1 \end{array}\right).$$

The following step is to apply local perturbations to each $f_k$. The aim of the next section is to define a family of functions $h_k$ that we will use to do the perturbation.



\section{ Proof of  Theorems \ref{teo1} and \ref{teo2}}

For a linear partially hyperbolic automorphim with eigenvalues $\lambda^s < \lambda^c < \lambda^u$,  corollary \ref{corocentral} implies that the following quantities are relevant to the amount of change of the central Lyapunov exponent after local perturbation method of Baraviera-Bonatti:
\begin{itemize}
\item $\alpha = \lambda^c / \lambda^u$;
\item ${\displaystyle C = \max_{x \in B_r} \frac{h^c_r}{h^u_r} \cdot \max_{x \in B_r}{|| \text{Proj} _{u}(e_c)||}}$;
\item $vol(B_r)$ and the return time $n_r$;
\item $I(h).$
\end{itemize}

We consider the family $f_k$ constructed in the previous section. Recall that for each $k$ there exists an adapted inner product (which gives adapted metric) where $e^s_k, e^c_k, e^u_k$ form an orthonormal set. As these eigenspaces are converging to the canonical basis the adapted metrics are close to the euclidean metric when $k$ is large enough. By euclidean metric we mean the usual inner product coming from euclidean inner product of $\mathbb{R}^3.$ We will take a non-fixed point $p$ ($f_k(p) \neq p.$) and $0 < r < 1$ we  take a local coordinate $\pi_k: B_r^k(p) \ar B_r(0) \subset \mathbb{R}^3$  such that the adapted inner product is the pullback by $D \pi_k$ of the euclidean inner product. Here $B_r^k$ is the ball of radius $r$ with respect to the adapted metric.
 
 Take an arbitrary point
$$p \in F:= P( (0,2/3) \times (0,1) \times (5/6,1))$$ where $P: \mathbb{R}^3 \ar \toro{3}$ is the usual canonical projection. 
 Take fixed small $r$ such that $B^k_{r} (p) \in F$ for large $k.$   It is possible to find such $r > 0$, because for large $k$ all the adapted metrics are close to the euclidean metric.
It is easy to see  that $f_k^{-1}(F) \cap F=\emptyset$ so $ f_k(F) \cap F =\emptyset$ and  $ f_k(B^k_{r_0}(p )) \cap B^k_{r_0}( p) = \emptyset$, $\forall k \geq k_0$. 
 
We take  $h: B_1(0) \ar B_1(0)$  such that $ 0 \neq \|h - Id\|_{C^1} < \eta$ ($\eta$ will be defined later).  Let $\pi_k : B^k_r(p) \ar B_r(0) \subset \mathbb{R}^3$ be a local coordinate which is isometry and the derivative of $\pi_k$ sends $e^s_k, e^c_k, e^u_k$ to the canonical basis of $\mathbb{R}^3$.  Let  $\xi : \reais{3} \ar \reais{3}$, $\xi_r(x):= \frac{1}{r}x$ be the homothety of ratio $\frac{1}{r}$, then we define $\varphi_{k,r}(x) := \xi_r \circ \pi_k (x)$ and like in section (\ref{LP}) we construct $h_{k, r} : \toro{3} \ar \toro{3}$ as follows: 
\begin{equation} \label{normalize2} h_{k,r}(w) = \twopartdef {w}{w \notin  B^k_{r}(p);}{ \varphi_{k,r} ^{-1} \circ h \circ \varphi_{k,r} (w) }{ w \in   B^k_{r}(p).  }   
\end{equation}

%

Now the idea is to consider the arcs of diffeomorphisms $f_{k,r} := f_k \circ h_{k,r}$ and show that for some positive $r$ and all large $k$, $h_{k, r}$ is partially hyperbolic with positive Lyapunov exponent almost everywhere.

We need to guarantee that these arcs are composed of partially hyperbolic diffeomorphisms.
 
Even knowing that the set of partially hyperbolic diffeomorphisms is open in $C^1$- topology, we do not know the ``size" of this set.  To construct our desired examples in Theorems \ref{teo1}, \ref{teo2}  we need the sequence $\{h_k\}$ to be ``far from" $Id$, so it is not obvious that the composition $f_{k, r} = f_k \circ h_{k,r}$ is partially hyperbolic.  However, in our case this is not a serious issue. As $k$ grows, the domination between invariant sub bundles of $f_k$ is getting better and  and the expansion and contraction of respectively expanding and contracting bundles increase. We observe that when the domination between bundles is bigger, one can take wider invariant cones in the definition of partial hyperbolicity by cones (see Appendix).  So it is reasonable to expect that  we can do bigger perturbations of $f_k$ and still remain in  the partially hyperbolic diffeomorphisms set.  
%


\begin{lema} \label{parcialmente}
Let $\{f_k\}$ be the sequence of linear partially hyperbolic automorphisms defined as before and $0 < r < 1.$ There exist  $\eta >0$ and $K_0$ such that if $h: B_1(0) \ar B_1(0)$ is a diffeomorphism satisfying $||h-Id||_{C^1} < \eta$ and equal to identity on a neighborhood of the boundary of $B_1(0)$ then,  for $k \geq k_0$, $f_k \circ h_{k, r}$ is absolutely partially hyperbolic. 
\end{lema}
\begin{proof}

In the Appendix, for any linear partially hyperbolic diffeomorphism $f$ we estimate the size of the $C^1$-neighborhood of  $f$ inside absolutely partially hyperbolic diffeomorphisms. Here we are dealing with a sequence $f_k$ and the claim is that a same estimate for the size of neighborhood works for all large enough $k.$  

Now by remark \ref{remarkimp} 
 the size of permitted perturbation (i.e the number $\varepsilon$) in the lemma \ref{parcialmente1} depends increasingly on the ratio $\Theta_k:= \min  \left\{ \frac{|\lambda^u_k|}{|\lambda^c_k|}, \frac{|\lambda^c_k|}{|\lambda^s_k|} \right\}$. When $k$ grows this ratio also grows. So we take the same $\varepsilon$ for all $f_k.$ 
We should emphasize that the size of permitted perturbation is measured in the distance corresponding to the adapted metric of $f_k.$

So let $\varepsilon$ be as above and take any $\eta \leq \varepsilon.$ We have $\|\xi_r^{-1}\circ h\circ\xi_r - Id\|_{C^1} \leq r \|h - Id\|_{C^1} \leq \|h - Id\|_{C^1} \leq \varepsilon.$ By definition of adapted metric for each $f_k$ we have $D \pi_k$ preserves norms and angles and consequently the distance (adapted norm corresponding to $f_k$) between $ (\pi_k \circ \xi_r)^{-1} \circ h \circ (\pi_k \circ \xi_r)$  and identity is also less than $\varepsilon$ and we can apply lemma \ref{parcialmente1} taking $g:=h_{k, r}.$
%
%
\end{proof}



By the above lemma  it follows that $f_{k,r}$ is partially hyperbolic for  large enough $k$. Also, since the same family of invariants cones works for  both $f_{k,r}$ and $f_k$, the angle between the new center bundle and $E^u_{f_{k,r}}$ is uniformly bounded. That is, the norm of the projection of $E^c_{f_{k,r}}$ over $E^u_{f_{k,r}}$ parallel to the new center bundle is uniformly bounded.

 Now let $n_r(k)$ be the least positive integer for which $(f_k)^{n_r(k)}(B^k_r(p)) \cap B^k_r(p) \ne \emptyset$. Then we have
$$\sigma^c_{f_{k,r}}- \log \lambda^c _k\geq vol(B^k_r(p)) \cdot (-I(h) - C_k \alpha_k^{n_r(k)})$$

where 
$${\displaystyle  \alpha_k = \frac{\lambda^c}{\lambda^u}, \text{ and }C_k= \max_{x\in B_r^k} \frac{h^c_{r, k}}{h^u_{r, k}} \cdot \max_{x \in B_r^k}{|| \text{Proj} _{u}(e_c)||} }.$$ 

As $\max_{x \in B_r^k}{|| \text{Proj} _{u}(e_c)||} $ is uniformly bounded, it follows that $C_k$ is uniformly bounded, say $C_k<D, \forall k$. Thus, since $n_r(k) \geq 2 $ for all $k$, we get:
$$\sigma^c_{f_{k,r}} - \log \lambda^c_k \geq vol(B^k_r(p)) \cdot (-I(h) - D \alpha_k^2).$$

Observe that  $\alpha_k \ar 0$ when $k \ar \infty$. So, for large values of $k$ we have
 $$-I(h_k) - D \alpha_k^2 \geq \frac{-I(h)}{2}$$ 

 which implies
$$\sigma^c_{f_{k, r}} - \log \lambda^c_k \geq -vol(B^k_r(p)) \cdot \frac{I(h)}{2} \ar -vol(B(r,p))\frac{I(h)}{2} > 0.$$

Observe that the volume appearing in the above equations is the fixed euclidean volume on the torus and as $k$ is large enough the volume of $B_r^k(p)$ is close to $B_r(0).$   Thus since $\log \lambda^c_k \ar 0$, for large values of $k$ we get 
$$\sigma^c_{f_{k, r}}- \log \lambda^c_k > -\log \lambda^c_k \Rightarrow \sigma^c_{f_{k,r}} > 0.$$

Here we conclude the proof of Theorem \ref{teo1}.  We have obtained $f_{k, r}$ isotopic to Anosov such that the average of central Lyapunov exponent is positive.  Indeed, $H_k:[0,1] \times \toro{3} \ar \toro{3}$, $H_k(s,x) =f_k \circ h_{k, sr} $ is an isotopy between $f_k$ and $f_{k,r}$. By continuity argument we conclude that there exists an open subset of volume preserving diffeomorphisms $U$ containing $f_{k, r}$ such that for any $g \in U$ we have $\sigma^c_g > 0.$ By the following result of Hammerlindl-Ures we conclude that all $g \in U$ are ergodic and so the central Lyapunov exponent of almost every point is positive.

\begin{teo}{\cite{HU}}\label{HamUres}
Let $f: \toro{3} \ar \toro{3}$ be a $C^{1+\alpha}$ volume preserving partially hyperbolic diffeomorphism, homotopic to a hyperbolic automorphism $A$. Assume $f$ is not ergodic. Then,
\begin{itemize}
\item $E^s \oplus E^u$ integrates to a minimal foliation;
\item $f$ is topologically conjugate to $A$ and the conjugacy carries strong leaves of $f$ to the correspondent strong leaves of $A$;
\item the central Lyapunov exponent of $f$ is $0$ almost everywhere.
\end{itemize}
\end{teo}

 Now we prove Theorem \ref{teo2}. Again from the continuity of $\sigma^c$, there is some $0<r_0 < r$ for which $\sigma^c_{f_{k, r_0}} = 0$.


That is, we got a partially hyperbolic diffeomorphism $g := f_{k,r_0}: \toro{3} \ar \toro{3}$, isotopic to an Anosov diffeomorphism and with $\sigma^c_g = 0$. However, using again the above Theorem we obtain the following:


\begin{cor}
The diffeomorphism $g$ obtained above has zero central Lyapunov exponent almost everywhere.
\end{cor}
\begin{proof}
Indeed, if $g$ is ergodic then the Lyapunov exponents are constant almost everywhere. So we have $\lambda^c_g =\sigma^c_g = 0$ almost everywhere. On the other case, that is, if $g$ is not ergodic then by the previous theorem (third item) we have that $\lambda^c_g = 0$ almost everywhere.
\end{proof}

To finish the proof we note that, since $g$ is isotopic to a linear Anosov automorphism, by \cite{H} all central leaves are homeomorphic to $\mathbb{R}.$  $\blacksquare$

%

\section{Appendix: Cone constructions and absolutely partially hyperbolic diffeomorphisms} \label{appendix}

\begin{defi}
Given an orthogonal splitting of the tangent bundle of $M$
$$E \oplus F = T M$$

and a real constant $\beta >0$, for each $x\in M$ we define the cone centered in $E(x)$ with angle $\beta$ as
$$C(E,x,\beta) = \{ v \in T_xM : ||v_F || \leq \beta ||v_E|| , \text{ where } v = v_E + v_F  , v_E \in E(x), v_F\in F(x).\}$$
\end{defi}

Given a partially hyperbolic diffeomorphism $f: \toro{3} \ar \toro{3}$ with invariant splitting
$$T M = E^s \oplus E^c \oplus E^u$$

there is an adapted inner product (and then an adapted norm) with respect to which the splitting is orthogonal (see \cite{YP}). Thus, given $\beta>0$ we can define standard families  of cones centered on the fiber bundles $E^{\tau}(x)$ with angle $\beta$, $C^{\tau}(x,\beta)$, $\tau = s,c,u,cs,cu$.

Consider $f:M \ar M$ a (absolutely) partially hyperbolic diffeomorphism. By using an adapted norm $|| \cdot ||$, we can consider the invariant splitting 
$$T M = E^s \oplus E^c \oplus E^u$$
as being an orthogonal splitting, and there exist numbers
$$0 < \lambda_1 \leq \mu_1 < \lambda_2 \leq \mu_2 <\lambda_3 \leq \mu_3, \hspace{.3cm} \mu_1<1, \hspace{.3cm} \lambda_3>1$$
for which
$$\lambda_1 \leq ||Df(x)|E^s(x)|| \leq \mu_1$$
$$\lambda_2 \leq ||Df(x)|E^c(x)|| \leq \mu_2$$
$$\lambda_3 \leq ||Df(x)|E^u(x)|| \leq \mu_3.$$

Partial hyperbolicity can be described in terms of invariant cone families (see \cite{YP}, pg.15). More specifically, let $f: \toro{3} \ar \toro{3}$ be a partially hyperbolic diffeomorphism and 
$$T _x\toro{3} = E^s(x) \oplus E^c(x) \oplus E^u(x)$$

a continuous orthogonal splitting of $T \toro{3}$. Given a real number $\beta>0$ define the families of cones
$$C^s(x,\beta) = C(x,E^s(x),\beta), C^u(x,\beta) = C(x,E^u(x),\beta)$$
$$C^{cs}(x,\beta) = C(x,E^{cs}(x),\beta),C^{cu}(x,\beta) = C(x,E^{cu}(x),\beta)$$

where $E^{cs}(x) = E^c(x) \oplus E^s(x), E^{cu}(x) = E^c(x) \oplus E^u(x).$

Then, $f$ is absolutely partially hyperbolic if, and only if, there is $0 < \beta < 1$ and constants
$$ 0<\mu_1< \lambda_2 \leq \mu_2 < \lambda_3 , \hspace{.3cm} \mu_1<1, \hspace{.3cm} \lambda_3>1$$
for which
\begin{equation}\label{2.11}
\begin{split}
Df^{-1}(x) (\cone{\tau}{}{x}{\beta}) & \subset  \cone{\tau}{}{f^{-1}(x)}{\beta}, \tau =s,cs ; \\
Df(x) (\cone{\Psi}{}{x}{\beta}) & \subset  \cone{\Psi}{}{f(x)}{\beta} ,\Psi=u,cu;
\end{split}
\end{equation}

and
\begin{equation}\label{2.12}
\begin{split}
||Df^{-1}(x) v|| & >  \mu_1^{-1}||v||, v \in C^{s}(x,\beta) ;\\
||Df^{-1}(x) v|| & >  \mu_2^{-1}||v||, v \in C^{cs}(x,\beta); \\
||Df(x) v|| & > \lambda_3 ||v||, v \in C^{u}(x,\beta) ;\\
||Df(x) v|| & >  \lambda_2||v||, v \in C^{cu}(x,\beta).
\end{split}
\end{equation}

For the linear case, we can get some more precise and interesting conclusions regarding the relation of the constants of (\ref{2.11}) and (\ref{2.12}) and the Lyapunov exponents of the function. In what follows, for a linear partially hyperbolic diffeomorphism, we find an explicit relation between the angle of the invariant cones families and the ratio of domination between unstable, stable and central bundles.

Consider $f: \toro{3} \ar \toro{3}$ a linear, volume preserving, partially hyperbolic diffeomorphism of $\toro{3}$. Denote by $\lambda^s,\lambda^c,\lambda^u$ its eigenvalues, where
$$|\lambda^s| < |\lambda^c| < |\lambda^u| , \hspace{.3cm} |\lambda^s| <1 < |\lambda^u|.$$ 

Put
\begin{equation} \label{THETA}
\Theta := \min \left\{ \frac{|\lambda^c|}{|\lambda^s|} , \frac{|\lambda^u|}{|\lambda^c|} \right\}
\end{equation}

Then we can choose a constant  $\beta > 0$ such that
\begin{equation}\label{BETA}
1< (1+\beta)^2 < \Theta.
\end{equation}

Therefore, by the definition of $\beta$, we have
$$ (1+\beta)|\lambda^s| < \frac{| \lambda^c|}{1+\beta} < (1+\beta) |\lambda^c| , \hspace{.3cm} (1+\beta)|\lambda^s| <1<\frac{|\lambda^u|}{1+\beta}.$$

Consequently we can find constants $\mu_1, \lambda_2, \mu_2, \lambda_3$ such that
\begin{equation} \label{ctes}
(1+\beta)|\lambda^s| < \mu_1 < \lambda_2 < \frac{| \lambda^c|}{1+\beta} < (1+\beta) |\lambda^c| < \mu_2 < \lambda_3 < \frac{|\lambda^u|}{1+\beta} , \hspace{.3cm} \mu_1 <1<\lambda_3.
\end{equation}

Now, it is straightforward to verify that with the constants defined by (\ref{BETA}) and (\ref{ctes}), the families of stable, unstable, center-stable and center-unstable cones satisfies (\ref{2.11}) and (\ref{2.12}). For example, if we take $v =v_s +v_{cu }\in C^{cu}(x,\beta)$ then
$$||Df(x)v_s|| = |\lambda^s| ||v_{s}|| \leq \beta |\lambda^s| ||v_{cu}|| < \beta|\lambda^c| ||v_{cu}|| \leq \beta ||Df(x)v_{cu}||$$

that is
$$Df(x) (C^{cu}(x,\beta)) \subset C^{cu}(f(x),\beta).$$

Furthermore,
$$||Df(x)v||^2 \geq ||Df(x)v_{cu}||^2 \geq |\lambda^c|^2 ||v_{cu}||^2.$$

But by (\ref{ctes}) we know that $|\lambda^c| > (1+\beta)\lambda_2$. So,

$$||Df(x)v||^2 > (1+\beta)^2 (\lambda_2)^2||v_{cu}||^2 \geq (\lambda_2)^2 (||v_{cu}||^2 + \beta^2 ||v_{cu}||^2) \geq \lambda_2^2 (\|v_{cu}\|^2 + \|v_s\|^2 )= (\lambda_2||v|| )^2$$

$$ \Rightarrow ||Df(x)v|| > \lambda_2 ||v||.$$

The argument for the other cones is similar.\\

\subsection{ Size of perturbation among absolutely partially hyperbolic diffeomorphisms}

\quad\\

Here we find an estimative for the size of the $C^1$-neighborhood of a linear partially hyperbolic automorphism inside absolutely partially hyperbolic diffeomorphisms.

\begin{lema} \label{parcialmente1}
Let $f: \toro{3} \ar \toro{3}$ be a linear partially hyperbolic diffeomorphism, volume preserving, with eigenvalues $\lambda^s, \lambda^c, \lambda^u$, where $|\lambda^s| < |\lambda^c| < |\lambda^u|$. Then, there is a constant $\varepsilon>0$   such that, for every diffeomorphism $g:\toro{3} \ar \toro{3}$ with $||g-Id ||_{C^1} < \varepsilon$ (adapted norm corresponding to $f$), the composition $f \circ g$ is an absolutely partially hyperbolic diffeomorphism.  The constant $\varepsilon$ depends only on  $\Theta:= \min \left\{ \frac{|\lambda^u|}{|\lambda^c|}, \frac{|\lambda^c|}{|\lambda^s|} \right\}$.
\end{lema}

\begin{remark} \label{remarkimp}
It is easy to see from the proof that $\varepsilon$ depends increasingly on $\Theta.$ Bigger $\Theta$ permits bigger perturbations. 
However we emphasize that the distance between $g$ and identity is considered in the adapted metric.  
\end{remark}
\begin{proof}
Let $f$ be as in the statement and denote the invariant splitting of $f$ by
$$T_xM = E^s(x) \oplus E^c(x) \oplus E^u(x).$$

Consider the adapted norm $||\cdot ||$ with respect to which the invariant splitting is orthogonal. Now, since $f$ is a linear  partially hyperbolic diffeomorphisms we can choose constants 
$$0 < \lambda_1 \leq \mu_1 < \lambda_2 \leq \mu_2 < \lambda_3 \leq \mu_3$$
$$ \mu_1 < 1 < \lambda_3$$
and a real value $\beta >0$ as in (\ref{BETA}) and (\ref{ctes}). 

For $v \in M$ we can write $v = v_s + v_c + v_u$ with $v_{\tau} \in E^{\tau}(x)$, $\tau=s,c,u$.

\begin{itemize}
\item If $v \in C^{u}(x,\beta)$ then $||v_{cs}|| \leq \beta ||v_u||$, where $v_{cs} = v_s + v_c$.
Thus,
$$||Df(x)v_{cs}|| < \mu_2 || Df^{-1} \circ Df(x) v_{cs}|| = \mu_2 ||v_{cs}|| \leq \mu_2 \beta ||v_u|| < \mu_2 \beta (\lambda_3)^{-1}||Df(x)v_u||.$$

$$\Rightarrow Df(x)v \in C^u(f(x),(\mu_2 / \lambda_3) \cdot \beta).$$

\item If $v \in C^{cu}(x,\beta)$ then $||v_s|| \leq \beta ||v_{cu}||$ where $v_{cu} = v_c + v_u$.
Thus,
$$||Df(x)v_{s}|| < \mu_1 || Df^{-1} \circ Df(x) v_{s}|| = \mu_1 ||v_{s}|| \leq \mu_1 \beta ||v_{cu}|| < \mu_1 \beta (\lambda_2)^{-1}||Df(x)v_{cu}||$$

$$ \Rightarrow Df(x) C^{cu}(x,\beta) \subset C^u(f(x),( \mu_1/\lambda_2 )\cdot \beta).$$

\item If $v \in E^s(x)$ then $||v_{cu}|| \leq \beta || v_s|| $. Thus,
$$ ||Df^{-1}(x)v_{cu}|| < \lambda_2^{-1} || Df\circ Df^{-1}(x) v_{cu}|| = \lambda_2^{-1} ||v_{cu}|| \leq \lambda_2^{-1} \beta ||v_s|| < \lambda_2^{-1} \beta \mu_1||Df^{-1}(x)v_s||$$

$$ \Rightarrow Df^{-1}(x) C^s(x,\beta) \subset C^{s}(f^{-1}(x),(\mu_1 / \lambda_2) \cdot \beta).$$

\item If $v\in E^{cs}(x)$ then $||v_u||\leq \beta||v_{cs}||$. Thus,
$$ ||Df^{-1}(x)v_{u}|| < \lambda_3^{-1} || Df\circ Df^{-1}(x) v_{u}|| = \lambda_3^{-1} ||v_{u}|| \leq \lambda_3^{-1} \beta ||v_{cs}|| < \lambda_3^{-1} \beta \mu_2 ||Df^{-1}(x)v_{cs}||$$

$$\Rightarrow Df^{-1}(x) C^{cs}(x,\beta) \subset C^{cs}(f^{-1}(x),(\mu_2 / \lambda_3) \cdot \beta).$$
\end{itemize}

Define
$$\gamma := \max \left\{\frac{\mu_2}{\lambda_3}, \frac{\mu_1}{\lambda_2} \right\} < 1.$$ 

So we have
\begin{eqnarray*}
Df^{-1}(x) (\cone{\tau}{}{x}{\beta}) & \subset & \cone{\tau}{}{f^{-1}(x)}{\gamma \cdot \beta }, \tau =s,cs ; \\
Df(x) (\cone{\Psi}{}{x}{\beta}) & \subset & \cone{\Psi}{}{f(x)}{\gamma \cdot \beta} ,\Psi=u,cu;
\end{eqnarray*}

Observe that by (\ref{BETA}) and (\ref{ctes}), $\beta$ and $\gamma$ depends only on the ratios $|\lambda^u| / |\lambda^c| , |\lambda^c| / |\lambda^s|$.

Now, since the invariant splitting is constant, we can take an $\varepsilon > 0$ depending only on the ratios $|\lambda^u| / |\lambda^c| , |\lambda^c| / |\lambda^s|$ such that, if $|| g -Id ||_{C^1} < \varepsilon$ then

\begin{equation}\label{hcone}
\begin{split}
Dg(x) C^{\tau}\left(x,\gamma \cdot \beta\right) & \subset C^{\tau}\left(g(x),\beta \right), \tau=u,cu \\
Dg^{-1}(x) C^{\Psi}\left(x,\gamma \cdot \beta\right) & \subset C^{\Psi}\left(g^{-1}(x), \beta\right), \Psi=s,cs
\end{split}
\end{equation}

and
$$\frac{l}{L} > \gamma  , \hspace{.3cm} L < \frac{1}{\mu_1},\hspace{.3cm} \frac{1}{\lambda_3} < l$$

with
$$l ||v|| \leq ||Dg(x)v|| \leq L||v||.$$

\begin{figure}
\includegraphics[scale=0.42]{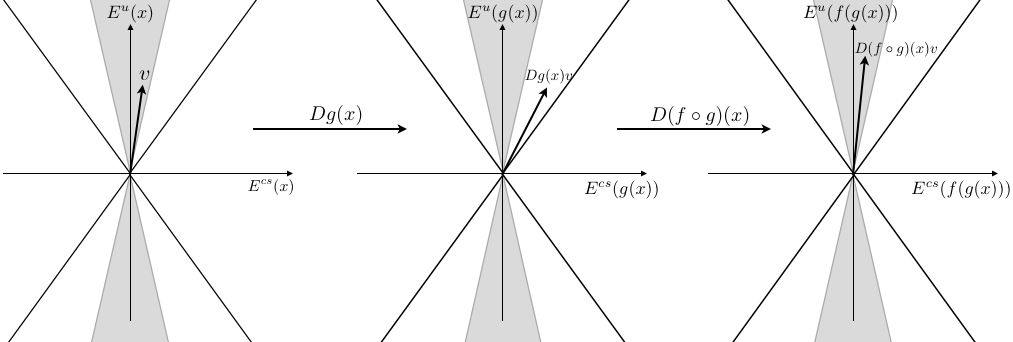}
\caption{The grey cones denotes the cones with angle $\gamma \cdot \beta$ and the wider ones are the cones with angle $\beta$.}
\end{figure}

Thus we have (see figure 1)
$$D(f\circ g)(x) (C^{\tau}(x,\gamma \cdot \beta)) \subset C^{\tau}(f(g(x)),\gamma \cdot \beta), \tau=u,cu ;$$
$$D(f \circ g)^{-1}(x) (C^{\Psi}(x,\beta)) \subset C^{\Psi}(g^{-1}(f^{-1}(x)),\beta), \Psi = s,cs.$$

Now, we need to show uniform contraction and expansion on these families.\\

\begin{itemize}
\item  If $v \in C^{u}(x,\gamma \cdot \beta)$ then
$$|| D(f \circ g ) (x) v|| \geq \lambda_3 ||Dg(x)v|| \geq \lambda_3 \cdot l \cdot ||v||. $$

\item If $v \in C^{cs}(x,\beta)$ then
$$||D(f\circ g)^{-1}(x) v|| \geq L^{-1} || Df^{-1}(x)v|| > L^{-1} \cdot \mu_2^{-1} ||v|| .$$

\item If $v \in C^{cu}(x,\beta\cdot \gamma)$ then
$$|| D(f \circ g ) (x) v|| \geq \lambda_2 ||Dg(x)v|| \geq \lambda_2 \cdot l \cdot ||v||. $$

\item If $v \in C^{s}(x,\beta)$ then
$$||D(f\circ g)^{-1}(x) v|| \geq L^{-1} || Df^{-1}(x)v|| > L^{-1} \cdot \mu_1^{-1} ||v|| .$$
\end{itemize}

Furthermore
$$0 < L \cdot \mu_1 < l \cdot \lambda_2 \leq L \cdot \mu_2 < l \cdot \lambda_3.$$

and
$$L \cdot \mu_1 < 1 \hspace{.3cm} , \hspace{.3cm}  l \cdot \lambda_3 >1,$$

so that $f \circ g$ is absolutely partially hyperbolic as we claimed.
\end{proof}


\bibliographystyle{plainyr}
\bibliography{Referencias}

\end{document}